\documentclass[12pt]{amsart}

\usepackage{amssymb}
\usepackage{amsmath}
\usepackage{amsthm}

\usepackage[english]{babel}

\numberwithin{equation}{section}

\newtheorem{thm}{Theorem}[section]
\newtheorem{pr}{Proposition}[section]
\newtheorem{lem}{Lemma}[section]
\newtheorem{cor}{Corollary}[section]

\newtheorem{df}{Definition}

\theoremstyle{definition}
\newtheorem{rem}{Remark}

\newcommand{\Nbb}{\mathbb{N}}
\newcommand{\Zbb}{\mathbb{Z}}
\newcommand{\lip}{\mathrm{Lip}}
\newcommand{\al}{\alpha}

\newcommand{\Rbb}{\mathbb R}
\newcommand{\om}{\omega}
\newcommand{\omt}{\widetilde{\omega}}
\newcommand{\bmo}{\mathrm{BMO}}

\newcommand{\QQ}{\mathfrak{Q}}

\begin{document}

\author{Andrei V.\ Vasin}
\address{Admiral Makarov State University of Maritime and Inland Shipping,
Dwinskaya Street 5/7, St.~Petersburg 198255, Russia}
\email{andrejvasin@gmail.com}

\author{Evgueni Doubtsov}
\address{St.~Petersburg Department of V.A.~Steklov Mathematical Institute,
Fontanka 27, St.~Petersburg 191023, Russia}
\email{dubtsov@pdmi.ras.ru}

\title{A T(P) theorem for Zygmund spaces on domains}

\begin{abstract}
Let $D\subset \mathbb{R}^d$ be a bounded Lipschitz domain,
$\omega$ be a high order modulus of continuity and let $T$ be a convolution Calder\'{o}n--Zygmund operator.
We characterize the bounded restricted operators
$T_D$ on the Zygmund space $\mathcal{C}_{\omega}(D)$.
The characterization is based on properties of $T_D P$
for appropriate polynomials $P$ restricted to $D$.
\end{abstract}

\thanks{This research was supported by the Russian  Foundation for Basic Research (grant No.~20-01-00209).}

\maketitle

\section{Introduction}

\subsection{Basic definitions}

\subsubsection{Restricted Calder\'{o}n--Zygmund operators}
A $C^{k}$-smooth  homogeneous Calder\'{o}n--Zy\-gmund
operator is a principal value convolution operator
\[
Tf(y)= PV \int_{\mathbb{R}^d} f(x) K(y-x)\, dx,
\]
where $dx$ denotes  Lebesgue measure in $\mathbb{R}^d$ and
\[
K(x) =\frac{\Omega(x)}{|x|^d },\quad x \neq 0;
\]
it is assumed that $\Omega(x)$ is a homogeneous function of degree $0$
and $\Omega(x)$ is $C^k$-differentiable on $\mathbb{R}^d \setminus \{0\}$
with zero integral on the unit sphere.
The function $K(x)$ is called a Calder\'{o}n--Zygmund kernel.

Given a domain $D \subset\mathbb{R}^d$,
we consider the corresponding modification of $T$.
Namely, the operator $T_D$ defined by the formula
\[
T_Df=  (Tf)\chi_D,\quad \mathrm{supp} f\subset \overline{D},
\]
is called  a \textit{restricted} Calder\'{o}n--Zygmund operator.

In the present paper, we study certain smoothness properties of $T_D$
for a domain $D$ with regular boundary.

\subsubsection{Lipschitz domains}

\begin{df}\label{df3}
A bounded domain $D\subset  \mathbb{R}^d$
is called $(\delta, R)$-Lipschitz if,
for every point
 $a\in\partial D$, there exists a function  $A:\mathbb{R}^{d-1}\rightarrow \mathbb{R}$
with $\|\nabla A\|_\infty\leq \delta$, and
there exists a cube  $\QQ\subset\mathbb{R}^d$
with side length $R$ and center $a$ such that
the equality
\[
D\cap \QQ= \left\{ (x,y)\in (\mathbb{R}^{d-1}, \mathbb{R})\cap \QQ:\ y>A(x) \right\}
\]
holds after a suitable shift and rotation of the coordinate system.
The cube $\QQ$ is called an $R$-window for the domain under consideration.
\end{df}

In what follows, the parameters
$\delta$ and $R$ are not explicitly specified.
We consider general Lipschitz domains, which does not lead to confusion.

Also, we use in the present paper standard Lipschitz spaces $\mathrm{Lip}_\al(D)$, $0<\alpha\le 1$. By definition, the space $\mathrm{Lip}_\al(D)$ consists of $f: D\to \Rbb$ such that
\[
\|f\|_{L^\infty(D)} + \sup_{x, y\in D,\, x\neq y} \frac{|f(x) - f(y)|}{|x-y|^{\alpha}} < \infty.
\]

\subsubsection{Zygmund spaces}
Following Janson \cite{Ja2}, we consider general moduli of continuity.

\begin{df}[see \cite{Ja2}]\label{df1}
A continuous increasing function $\omega:[0,\infty)\rightarrow[0,\infty)$, $\omega(0)=0$,
 is called a modulus of continuity of order $n\in\Nbb$ if $n$ is the smallest positive integer such that   the following two regularity properties are satisfied:

1. For some $q$, $ n\leq q<n+1$, the function $\frac{\om(t)}{t^q}$
is \textit{almost decreasing}, that is, there exists a positive constant $C=C(q)$ such that
  \begin{equation}\label{eq:eq42}
 \omega(st)<Cs^q\omega(t),\;s>1.
\end{equation}

2. For any $r$, $n-1<r< n$, the function
$\frac{\om(t)}{t^r}$ is \textit{almost increasing},
that is, there exists a positive constant $C=C(r)$ such that
     \begin{equation}\label{eq:eq32}
 \omega(st)<Cs^r\omega(t), \;s<1.
\end{equation}
\end{df}

In the studies of Zygmund spaces, we use the term cube and the notation $Q$
for a cube in the space $\mathbb{R}^d$
with edges parallel to the coordinate axes.
Note that no such restriction is imposed on the cube $\QQ$ in Definition~\ref{df3}.
Let $|Q|$ denote the volume of the cube under consideration and let $\ell=\ell(Q)$ denote its side length.
Let $\mathcal{P}_n$ denote the space of polynomials of degree at most $n$.

  \begin{df}
 \label{df2}
Given a modulus of continuity $\omega$ of order $n \in \mathbb{N}$, the homogeneous Zygmund space $\mathcal{C}_{\om}(D)$ in a domain $D\subset  \mathbb{R}^d$ consists of those $f\in L^1_{loc}(D,dx)$ for which the Campanato type seminorm
\begin{equation}\label{eq:eq1}
   \|f\|_{\omega,D}=\sup_{Q\subset D} \inf_{P\in\mathcal{P}_n}\frac{1}{\omega(\ell)} \|f-P\|_{L^1(Q, dx/|Q|)}
\end{equation}
is finite.
\end{df}

\begin{rem}
Classical arguments based on the Calder\'{o}n--Zygmund lemma and used in the studies of the
standard space $\mathrm{BMO}(\mathbb{R}^d)$ and Lipschitz spaces $\mathrm{Lip}_\al(\mathbb{R}^d)$
(see, for example, \cite{C, M} and \cite[Sec.\,1.2]{KK}) allow to verify that the
$L^1$-norm in definition (\ref{eq:eq1}) is replaceable by
the $L^p$-norm, $1 < p\leq\infty$, in an arbitrary domain $D$.
The corresponding seminorms are equivalent and define the same space.
See Proposition~\ref{pr4} in Section 2 for further details and proofs.
\end{rem}

\subsection {T(1) and T(P) theorems}
For general moduli of continuity of order $n$, $n\in\Nbb$, Janson \cite[Sec.\,6]{Ja2}
proved that the homogeneous spaces $\mathcal{C}_{\omega}(\mathbb{R}^d)$
are invariant under certain Fourier multipliers.
The spaces $\mathcal{C}_{\omega}(\mathbb{R}^d)$ considered in \cite{Ja2}
are defined in terms of finite differences; in the present paper, we use polynomial approximation.
 Also, for domains, it is natural to consider the corresponding inhomogeneous spaces.
Indeed, for a bounded Lipschitz domain $D$, the set
$\mathcal{C}_{\omega}(D)$ is contained in the space $L^1(D,dx)$.
So, by definition, the inhomogeneous space $\mathcal{C}_{\omega}(D)$
is a Banach space with the following norm:
\[
\|f\| =\|f\|_{\omega,D}+\|f\|_{L^1(D,dx)}.
\]

The present paper is motivated by a T(1) theorem used in the proof of the
the following result by Mateu, Orobitg and Verdera  \cite[Main Lemma]{MOV}
in the setting of the Lipschitz spaces on domains $D\subset\mathbb{R}^d$.

\begin{thm}[{\cite[Main Lemma]{MOV}}; see also \cite{An}]\label{thm1}
Let $D$ be a bounded domain with $C^{1+\alpha}$-smooth boundary, $0<\al<1$.
Then the restricted Calder\'{o}n--Zygmund operator $T_D$ with an even kernel
maps the Lipschitz space $\lip_\alpha(D)$ into itself.
\end{thm}

A related T(1) theorem for Hermit--Calder\'{o}n--Zygmund operators is proven in \cite{BCFST13}.
Theorem~\ref{thm1} is extended in \cite{V1} to weakly smooth spaces between
$\lip_\alpha(D)$ and $\bmo(D)$, that is,
the integer order $n=0$ is considered.

Observe that Theorem~\ref{thm1} is not only of independent interest,
but also has interesting and important applications.
In particular, Theorem~\ref{thm1} is used in \cite{MOV} to obtain
results on regularity of quasi-regular functions,
i.e., solutions of the Beltrami equation on the complex plane.
Further development of this topic is related to the regularity
of solutions to second-order elliptic equations in divergent form.
Also, Theorem~\ref{thm1} is combined in \cite{MOV} with results by
Tolsa~\cite{T05} to establish
a direct relation between removable sets
for the bounded quasi-regular functions and bounded holomorphic functions.

Next, let $\mathcal{P}_{n}(D)$ denote the space of polynomials
from $\mathcal{P}_{n}$ multiplied by the characteristic function of the domain $D$.
In this paper, higher orders of smoothness are considered.
So, we are also motivated by the following result of Prats and Tolsa \cite{PT}.

\begin{thm}[{\cite[Theorem 1.6]{PT}}]\label{thm2}
Let $D$ be a Lipschitz domain,
{$T_D$} be a restricted $C^n$-smooth convolution Calder\'{o}n--Zygmund operator,
$n\in\mathbb{N}$ and $p > d$. Then the operator $T_D$ is bounded on the Sobolev space $W^{n,p}(D)$
if and only if
$T_D P\in W^{n,p}(D)$ for any polynomial $P\in\mathcal{P}_{n-1}(D)$.
\end{thm}

By analogy with T(1) theorems,
Prats and Tolsa \cite{PT} refer to the above theorem as a T(P) theorem to indicate
explicitly that the corresponding characterization uses values of the operator
$T$ on the polynomials of appropriate degree.
Note that the kernel of the operator under consideration in Theorem~\ref{thm2} is not assumed to be even.
Also, it is shown in~\cite{PT} that Theorem~\ref{thm2} implies
regularity results, in terms of Sobolev spaces, for solutions of the Beltrami equation.

In the present paper, we obtain a similar T(P) result for the Zygmund spaces.

\subsection{Main theorem}
Given a modulus of continuity $\omega$,
the associated modulus of conti\-nuity $\widetilde{\omega}$ is defined as follows:
 \begin{equation}
 \label{eq:eq2}
   \widetilde{\omega}(x)= \frac{\omega(x)}{\max \left\{1,\int_x^1 \omega(t)t^{-n-1}dt\right\}}.
    \end{equation}

\begin{thm}\label{thm3}
Let $\omega$ be a modulus of continuity of order $n\in \mathbb{N}$
and let $D\subset \mathbb{R}^d$ be a bounded Lipschitz domain.
Let $T$ be a homogeneous
 $C^{n+1}$-smooth Calder\'{o}n--Zygmund operator.
Then the restricted operator $T_D$ is bounded on the space $\mathcal{C}_{\om} (D)$
if and only if two following properties hold:
\begin{enumerate}
  \item[(i)]  $T_D P \in \mathcal{C}_{\omega} (D)$ for any polynomial $P\in\mathcal{P}_n(D)$;
  \item[(ii)] for any cube $Q \subset D$ centered at $x_0$
  and for any polynomial $P_{x_0}$,
  homogeneous of degree $n$ with respect to $x-x_0$, there exists a polynomial
  $S_Q\in\mathcal{P}_n(D)$ such that
\[
     \|T_D(\chi_D P_{x_0}) - S_Q\|_{L^1(Q, dx/|Q|)}\leq C \|P\| \omt(\ell(Q))
\]
with a constant $C$ independent of $Q$.
\end{enumerate}
\end{thm}

It is worth mentioning that Theorem~\ref{thm3} applies to the classical Zygmund spaces
$\mathcal{Z}_n(D) := \mathcal{C}_{\om_n}(D)$, where $\om_n(t) = t^n$, $n\in \Nbb$.

\begin{cor}\label{cor3}
Let $n\in \mathbb{N}$ and $D\subset \mathbb{R}^d$ be a bounded Lipschitz domain.
Let $T$ be a homogeneous $C^{n+1}$-smooth Calder\'{o}n--Zygmund operator.
Then the restricted operator $T_D$ is bounded
on the Zygmund space $\mathcal{Z}_{n}(D)$
if and only if two following properties hold:
\begin{enumerate}
  \item[(i)]  $T_D P \in \mathcal{Z}_{n} (D)$ for any polynomial $P\in\mathcal{P}_n(D)$;
  \item[(ii)] for any cube $Q \subset D$ centered at $x_0$
  and for any polynomial $P_{x_0}$,
  homogeneous of degree $n$ with respect to $x-x_0$, there exists a polynomial
   $S_Q\in\mathcal{P}_n(D)$ such that
 \[
     \|T_D(\chi_D P_{x_0}) - S_Q\|_{L^1(Q, dx/|Q|)}\leq C \|P\|
     \frac{\ell^n}{\max \left\{1, \log\frac{1}{\ell} \right\}}
 \]
 with a constant $C$ independent of $Q$.
\end{enumerate}
\end{cor}

\begin{rem}
A modulus of continuity $\omega$ of order $n$ is called Dini regular if the integral
 \[\int_0 \omega(t)t^{-n-1}dt
 \]
 converges.
In this case, the functions
$\omega(x)$ and $\widetilde{\omega}(x)$ equivalent.
Therefore, the formulation of Theorem~\ref{thm3} essentially simplifies
and becomes a typical T(P) theorem: property~(ii) is superfluous,
since it follows from property~(i).
In the general setting, the functions $\omega(x)$ and  $\widetilde{\omega}(x)$ are not equivalent,
and property~(ii) based on $\widetilde{\omega}(x)$, in general,
does not follow from property~(i).
\end{rem}

\begin{rem}
If the functions $\om$ and $\widetilde{\omega}$ are not equivalent, then Theorem~\ref{thm3} becomes
asymmetric, in a sense.
Indeed, the space $\mathcal{C}_{\om}(D)$ is defined in terms
of $\om$,
however, property~(ii) from Theorem~\ref{thm3} is based on the modulus of continuity $\widetilde{\omega}$.
In particular, Corollary~\ref{cor3} illustrates such an asymmetry.
\end{rem}

\begin{rem}
Nonequivalent moduli of continuity may have
equivalent associated moduli of continuity. For instance,
we have $\omt_{s}\approx\omega_{-1}$
for the family of moduli of continuity $\om_{s}(t)=t \log^{s}1/t$, $s>-1$.
The operator $T_D$ is bounded on the space $\mathcal{C}_{\om_s} (D)$ for all
$s>-1$ if and only if it is bounded for some $s>-1$.
In fact, the family $\om_{s}(t)=t \log^{s}1/t$, $s\in\mathbb{R}$,
and the corresponding scale of Zygmund spaces $\mathcal{C}_{\om_{s}}(D)$
may serve as a useful working example for Theorem~\ref{thm3}.
These spaces reflect specific properties of the Zygmund scale in comparison with the Lipschitz one.
\end{rem}

\begin{rem}\label{rem11}
Let us explain our choice of
$\mathcal{P}_n$ as the approximating polynomial space in equality~(\ref{eq:eq1}).
Let $\mathcal{C}_{\omega, k}(D)$
denote the space generated by Definition~\ref{df2}
after replacement of $\mathcal{P}_n$ by the space $\mathcal{P}_k$.

$\bullet$ If $k>n$, then  the Marchaud type inequality for local polynomial approximations
(see, for example, \cite[Ch.\,4]{KK} for the power moduli of continuity)
guarantees that the corresponding seminorm defined by (\ref{eq:eq1}) generates the same space $\mathcal{C}_{\omega}(D)$,
up to factoriza\-tion by the polynomial space
$\mathcal{P}_k$.

$\bullet$ If $\om(t)=o(t^n)$, then for $k<n$, the space $\mathcal{C}_{\omega, k}(D)$
is trivial and coincides with the space of approximating polynomials
$\mathcal{P}_k(D)$.

$\bullet$ If $t^n=O(\om(t))$, then the value $k=n-1$ is admissible and
generates the scale of the Lipschitz--Bernstein spaces
$\mathcal{C}_{\omega, n-1}(D)$;
the standard Lipschitz space $\mathrm{Lip}_1(D)$
corresponds to the modulus of continuity $\om(t)=t$.
The scale of the spaces $\mathcal{C}_{\omega, n-1}(D)$
and that of the Zygmund spaces $\mathcal{C}_{\om}(D)$ are different.
In the present work, the spaces $\mathcal{C}_{\omega, n-1}(D)$
are not considered, since they are not
invariant under the convolution Calder\'{o}n--Zygmund
operators even in the case $D=\mathbb{R}^d$.
\end{rem}

\subsection {Notation and organization of the paper}
In Section 2, we introduce basic facts about the space $\mathcal{P}_n(D)$
and prove certain basic properties of the Zygmund spaces on domains.
The proof of the T(P) theorem is given in Section 3.

As usual, the letter $C$  denotes a constant, which may change from line to line
and does not depend of the relevant variables under consideration.
Notation $A\lesssim B$ means that there is a fixed positive constant
$C$ such that $A <C B$.
If $A\lesssim B\lesssim A$, then we write $A\approx B$
and we say that $A$ and $B$ are equivalent.

\section{Auxiliary results}

\subsection{Moduli of continuity and approximating polynomials}
A given modulus of continuity $\om$ is replaceable by an equivalent $C^\infty$-smooth
modulus of continuity $\widehat{\om}$ with the same natural parameter $n$.
Thus, in what follows, we assume that $\om$ is a $C^\infty$-smooth function on the ray $(0,\infty)$.

\begin{lem}[see, for example,  {\cite[Lemma~4]{Ja2}}]
\label{lem21}
Let $\om$ be a modulus of continuity.
Property (\ref{eq:eq42}) with parameter $q$ for the function $\om$ implies the estimate
  \begin{equation}
  \label{eq: eq22}
 \int_{t}^{\infty}\om(s)s^{-p-1}ds\lesssim \om(t)t^{-p},\quad p>q.
\end{equation}
Property (\ref{eq:eq32}) with parameter $r$ for the function $\om$ implies the estimate
\begin{equation}
\label{eq: eq21}
 \int_{0}^{t}\om(s)s^{-p-1}ds \lesssim \om(t)t^{-p},\quad p<r.
\end{equation}
 \end{lem}

Since any two norms on the space $\mathcal{P}_n$ are equivalent, the following lemma holds.

\begin{lem}[see~{\cite{DVS, KK}}]
\label{lem1}
Let $Q$ be a cube in $\mathbb{R}^d$ with center $x_0$ and side length $\ell$,
 $P= \sum_{|k|=0}^{n}a_k(x-x_0)^k$ be a polynomial on $\mathbb{R}^d$,
where $k=(k_1,\dots,k_d)\in \Zbb_+^d$ is a multiindex, $|k|=|k_1|+\dots+|k_d|$.
Then
 \[\sup_{x\in Q}|P(x)| \leq \sqrt{n} \sum_{|k|=0}^{n}|a_k| \ell^{|k|}
 \le C(n,d) \frac{1}{|Q|}\int_Q |P(x)|dx.
 \]
\end{lem}

Lemma~\ref{lem1} implies the following
lemma, where the notation $\ell_i=\ell(Q_i)$ is used for $Q_i$.

\begin{lem}
\label{lem2}
Let $Q_1\subset Q_2$ be two cubes in the space $\mathbb{R}^d$.
For every polynomial $P\in \mathcal{P}_n$, the following estimate holds:
 \[\|P\|_{L^1(Q_2, dx/|Q_2|)} \le C(n,d) \left(\frac{\ell_2}{\ell_1}\right)^n \|P\|_{L^1(Q_1, dx/|Q_1|)}.
 \]
\end{lem}

Given a cube $Q$ and $s>0$, let $sQ$ denote
the cube whose center coincides with the center of $Q$ and whose side length is equal to $s\ell(Q)$.

 \begin{df}
 \label{df22}
Let $f\in\mathcal{C}_{\om}(D)$ and $Q\subset D$ be a cube.
We say that $P_Q \in \mathcal{P}_n$ is a polynomial
of \textit{near best} approximation for the function $f$ on the cube $Q$ if
 \[\|f-P_Q\|_{L^1(Q, dx/|Q|)}\le C \om(\ell)\|f\|_{\om,D},
 \]
where the constant $C>0$ does not depend on $f$ and $Q$.

To extend functions from $\mathcal{C}_{\om}(D)$ to the entire space $\mathbb{R}^d$,
we consider the auxiliary space $\mathcal{C}^{int}_{\om}(D)$.
Namely, for $f\in L^1_{loc}(D)$, the corresponding
norm is defined by the following equality:
\begin{equation}
\label{eq:eq3}
   \|f\|^{int}_{\omega,D}=\sup_{Q:\, 2Q\subset D} \inf_{P\in\mathcal{P}_n}\frac{1}{\omega(\ell)} \|f-P\|_{L^1(Q, dx/|Q|)}.
\end{equation}
Similarly, for a function $f\in\mathcal{C}^{int}_{\om}(D)$, one introduces the polynomials
 of \textit{near best} approximation on cubes.
\end{df}

In both cases, the required approximating polynomials
can be selected in a certain unified way
(see~\cite{M}).
The construction below follows the corresponding argument from \cite[Ch.~1]{KK}.
Put $Q_0 = [-1/2, 1/2]^n$.
Let $\mathbb{P}$ be
an arbitrary projector from $L^1(Q_0,dx)$ onto $\mathcal{P}_n$.
Since $\mathcal{P}_n$ is a finite dimensional space, the operator $\mathbb{P}$ is bounded on
$L^p(Q_0,dx)$, $1 \leq p \leq \infty$.
Using a shift and a dilation, we transplant
the operator $\mathbb{P}$ to an arbitrary cube $Q$.
The norm of the resulting projector $\mathbb{P}_Q$ on $L^p(Q; dx/|Q|)$ does not depend on $Q$.
In particular, we obtain
\[\|\mathbb{P}_Q(f)\|_{L^\infty(Q)}\lesssim \frac{1}{|Q|}\int_Q |f|
\]
with a constant independent of $Q$ and $f$.
Next, for an arbitrary polynomial $u \in\mathcal{P}_n$, we have
$\mathbb{P}_Q(f-u)=\mathbb{P}_Q(f)-u$, hence,
\[\|\mathbb{P}_Q(f)-u\|_{L^{\infty}(Q)} \lesssim \|f-u\|_{L^1(Q, dx/|Q|)}.
\]
Therefore, in what follows, we assume that
$P_Q=\mathbb{P}_Q(f)$ is a polynomial of near best approximation on
$Q$ in any $L^p(Q)$-metric, $1 \leq p \leq \infty$.

\begin{lem}
\label{lem3}
Let $Q_1\subset Q_2\subset 4 Q_1$ be cubes in $D$
and let $P_{Q_1}, P_{Q_2} \in \mathcal{P}_n$ be polynomials of near
best approximation for $f\in \mathcal{C}_\omega(D)$ on the cubes $Q_1$ and $Q_2$,
respectively.
Then
\begin{equation}\label{eq:eq6}
\|P_{Q_1}-P_{Q_2}\|_{L^1(Q_2, dx/|Q_2|)}\le C(n,d)\omega(\ell_2)\|f\|_{\omega, D}.
\end{equation}
A similar lemma holds for $f\in \mathcal{C}^{int}_\omega(D)$, with appropriate changes.
\end{lem}
\begin{proof}
The analogue of estimate \eqref{eq:eq6} for $\|P_{Q_1}-P_{Q_2}\|_{L^1(Q_1, dx/|Q_1|)}$
holds by the triangle inequality.
Application of Lemma~\ref{lem2} finishes the proof.
\end{proof}

\subsection {Whitney coverings}
Fix a dyadic grid of semi-open cubes in $\mathbb{R}^d$.
 \begin{df}
\label{df44}
A collection of cubes $\mathcal{W}$ is called a Whitney covering of a Lipschitz domain $D$ if
the following conditions are fulfilled.
\begin{enumerate}
 \item [(i)] The collection $\mathcal{W}$ consists of dyadic cubes.
 \item [(ii)] The cubes from $\mathcal{W}$ are pairwise disjoint.
 \item [(iii)] The union of the cubes in $\mathcal{W}$ is $D$.
 \item [(iv)] $\mathrm{diam}(Q)\leq \mathrm{dist}(Q, \partial D) \leq 4\mathrm{diam}(Q)$.
 \item [(v)] If $Q$ and $R$ are neighbor cubes
 (i.e., $\overline{Q}\cap \overline{R}\neq \varnothing$), then $\ell(Q)\leq 4\ell(R)$.
 \item [(vi)] The family $\{\frac{6}{5}Q\}_{Q\in\mathcal{W}}$
 has finite superposition, i.e.,
      \[
      \sup_D\sum_{Q\in\mathcal{W}}\chi_{\frac{6}{5}Q}<\infty.
      \]
\end{enumerate}
\end{df}
Such coverings are well
known in the literature and widely used
(see \cite[Ch.\,6]{S}).

Each $R$-window $\QQ$
induces a vertical direction,
given by the eventually rotated $x_d$ axis.
The following property
easily follows (see \cite[Sec.\,3]{PT}) from the above properties
and the fact that the domain under consideration is Lipschitz:

\textrm{(vii)}
\textit{The number of Whitney cubes with the same side length, intersecting
a given vertical line in a window, is uniformly bounded. The corresponding vertical direction is the one induced by the window.
This is the last property of the Whitney cubes we need in what follows.
}

In fact, we need a Whitney covering $\mathcal{W}$ for a Lipschitz domain $D$ as well as a Whitney covering $\mathcal{W}'$ for its complement $D'=\mathbb{R}^d\setminus \overline{D}$.

\subsection{Extension of functions from domain to the entire Euclidean space}
To prove the following result, it suffices
to repeat the arguments used in the proof of Proposition~B.1 from \cite{V1}.

\begin{pr}
\label{pr1}
Let $\om(t)$ be a modulus of continuity of order $n \in \mathbb{N}$ and
$D\subset \mathbb{R}^d$ be a bounded Lipschitz domain.
Then the set $\mathcal{C}^{int}_{\om} (D)$ is contained in the space $L^1(D)$.
\end{pr}

Observe that Proposition~\ref{pr1} implies that one may equip
the inhomogeneous space $\mathcal{C}^{int}_{\om}(D)$ with the following norm:
\[
\|f\|=\|f\|^{int}_ {\omega,D}+\|f\|_{L^1(D,dx)}.
\]
We have $\mathcal{C}_{\om}(D)\subset\mathcal{C}^{int}_{\om}(D)$, thus,
\[
\|f\|=\|f\|_ {\omega,D}+\|f\|_{L^1(D,dx)}
\]
is a norm on the space $\mathcal{C}_{\om}(D)$.

If the modulus of continuity
$\om(t)$ is Dini regular, then arguments
from the monograph by Stein \cite[Ch.\,VI]{S} are applicable
for construction of an extension to the entire set $\mathbb{R}^d$.
In the general setting, we apply the approach used
by Jones \cite{Jo} for the space $\mathrm{BMO}$ and by
DeVore and Sharpley \cite{DVS} for the Besov spaces on uniform domains.
In particular, we need the following lemma.

\begin{lem}[{\cite[Lemma~5.2]{DVS}}]\label{lem5}
Let $D$ be a Lipschitz domain and $f\in L^1(D, dx)$.
Then there exist three positive constants $C$, $c$, $r_0$
depending only on the Lipschitz constants of $D$ and having the following property:
if $Q$ is a cube in
$\mathbb{R}^d$ with $\ell(Q)< r_0$
and such that $2Q\cap \partial D\neq \varnothing$, then
   \[
   \int_Q|\widetilde{f}-\widetilde{P}_Q|dx\leq C \sum_{S\subset cQ,\:S\in \mathcal{W}}\int_{S'}|f-P_{S'}|dx,
   \]
where $\widetilde{P}_Q$ is an appropriate polynomial in $\mathcal{P}_n$
and $S'=\frac{9}{8}S$ for each cube $S$.
\end{lem}

To obtain the required extension, we first fix a $C^{\infty}$-smooth partition
of unity $\{\psi_Q\}_{Q\in \mathcal{W}'}$
associated with a Whitney covering $\mathcal{W}'$ for $D'=\mathbb{R}^d\setminus \overline{D}$.
By definition, this means that the functions $\psi_Q$ have the following properties:
$\psi_Q$ is $C^{\infty}$-smooth,
$\chi_{\frac{4}{5}Q}\leq \psi_Q\leq \chi_{\frac{5}{4}Q}$, $Q\in \mathcal{W}'$, and
$\sum_{Q\in \mathcal{W}'} \psi_Q =\chi_{D'}.$

Given a Whitney cube $Q\in \mathcal{W}'$,
we say that a  Whitney cube $\widetilde{Q}\in \mathcal{W}$ is \textit{reflective}
to $Q$ provided that $\widetilde{Q}$ is a maximal cube such that
 $\textrm{dist}(Q, \widetilde{Q})\leq 2 \textrm{dist}(Q, \partial D)$.
Let $P_{\widetilde{Q}}$ denote a polynomial of near best approximation for
$f$ on the cube $\widetilde{Q}$.

Define an extension of $f$ as follows:
\begin{equation}\label{eq:eq222}
 \widetilde{f}=f\chi_D+\sum_{Q\in \mathcal{W}',\:\ell(Q)\leq R} \psi_Q P_{\widetilde{Q}},
\end{equation}
where $R$ is the Lipschitz constant from Definition~\ref{df3}.

\begin{pr}\label{pr3}
Let $\om$ be a modulus of continuity of order $n\in\mathbb{N}$ and $f \in \mathcal{C}^{int}_{\om}(D)$.
Then the function $\widetilde{f}$ defined by equality (\ref{eq:eq222})
has the following properties:
\begin{itemize}
  \item [(i)] the support of $\widetilde{f}$ is compact;
  \item [(ii)] the function $\widetilde{f}$ is $C^\infty$-smooth in the domain
   $D'=\mathbb{R}^d\setminus \overline{D};$
  \item [(iii)] $\widetilde{f} \in L^1(\mathbb{R}^d,dx)\cap\mathcal{C}_{\omega}(\mathbb{R}^d)$ and
     \[
     \|\widetilde{f}\|_{\omega,\mathbb{R}^d}+\|\widetilde{f}\|_{L^1(\mathbb{R}^d,dx)}\lesssim\|f\|^{int}_ {\omega,D}+\|f\|_{L^1(D,dx)}.
     \]
\end{itemize}
\end{pr}
\begin{proof}
  Properties (i) and (ii) are clear.
To prove (iii), we have to estimate the supremum on the right hand side of equality (\ref{eq:eq1})
for $D=\mathbb{R}^d$.

Firstly, we obtain the required estimates only for the cubes
$Q$ such that
$2Q\cap \partial D\neq \varnothing$ and $\ell(Q)< r_0$ for an appropriate parameter $r_0>0$.
Namely, we fix so small $r_0<R$ that Lemma~\ref{lem5}
holds for $r_0$, $c$, $C$, and the cube $cQ$ is contained in an $R$-window.

For any Whitney cube $S\in \mathcal{W}$, we have $ 2S'\subset D$, thus, Lemma~\ref{lem5} guarantees that

\[ I=\int_Q|\widetilde{f}-\widetilde{P}_Q|dx\leq C\|f\|_{\omega,D}^{int} \sum_{S\subset cQ,\:S\in \mathcal{W}}\omega (\ell(S)) \ell(S)^d.
\]

Now, we estimate the number of cubes of the same size in the above sum.
Property~(vii) of Whitney cubes, formulated after Definition~\ref{df44}, guarantees that, for every
Whitney cube $S\subset cQ$,
there exists a vertical line
(it is defined by the axis $x_d$ of the corresponding $R$-window),
which intersects finitely many  Whitney cubes with side length $\ell(S)$.
The number of the corresponding cubes is estimated above
by a constant $C$ depending only on the Lipschitz constants of the domain $D$.
Thus,
  \[
  \sharp{S}\lesssim \left(\frac{\ell(cQ)}{\ell(S)}\right)^{d-1},
  \]
where $\sharp{S}$ denotes the number of all cubes with side length
  $\ell(S)$ intersecting the cube $cQ$.

Let $s$ be the integer such that $2^s= \ell(S)$ and
let $m$ be the integer such that
$2^m\leq \ell(cQ)<2^{m+1}$.
Then
 \[\sharp{S}\lesssim\left(\frac{2^m}{2^s}\right)^{d-1}
 \]
with a constant independent of $Q$.
Since $\om$ is an increasing function, we obtain
\begin{align*}
 I
 &\lesssim\sum_{s=-\infty}^m\left(\frac{2^m}{2^s}\right)^{d-1}\omega (2^s) (2^s)^d
  \|f\|^{int}_ {\omega,D}
 =(2^m)^{d-1}\sum_{s=-\infty}^m{2^s}\omega (2^s) \|f\|^{int}_ {\omega,D} \\
 &\lesssim(2^m)^{d-1}\omega (2^m)\sum_{s=-\infty}^m{2^s} \|f\|^{int}_ {\omega,D}
  \lesssim(2^m)^{d}\omega (2^m) \|f\|^{int}_ {\omega,D}
  \lesssim |Q| \omega (\ell(Q))\|f\|^{int}_ {\omega,D},
\end{align*}
hence, we have the desired estimate for the small cubes located near the boundary
of the domain.

Next, if $\ell(Q)< r_0$ and $2Q\cap \partial D = \varnothing$,
then the required estimate for the supremum follows from the property
$\widetilde{f} \in C^\infty (D')$.

Finally, to prove the desired estimate for $\ell(Q)\ge r_0$, it suffices to show that
 $\widetilde{f}\in L^1(\mathbb{R}^d, dx)$.
The latter property follows from Proposition~\ref{pr1} and formula~\eqref{eq:eq222}.
The proof of the proposition is finished.
\end{proof}

\subsection{Equivalence of seminorms on Zygmund spaces}
Let $\om$ be a modulus of continu\-ity of order $n\in \Nbb$.
To prove the desired equivalence for different values of the parameter
$p$, $1\le p \le \infty$,
one may apply ideas from \cite{C, M}; see also \cite[Proposition~A.1]{V1}.

We need the following Calder\'{o}n--Zygmund lemma.
\begin{lem}[{\cite[Ch.\,1]{KK}}]
\label{lem6}
Let $Q$ be a cube, $f \in L^1(Q)$ and $A>\frac{1}{|Q|} \int_Q|f|$.
Then there exists an at most countable family $\{Q_i\}$
of dyadic cubes with disjoint interiors such that
   \begin{itemize}
     \item[(i)] $|f|\leq A$ a.e.\ on $Q\setminus\bigcup Q_i;$
   \item[(ii)] $A\leq 1/|Q_i| \int_{Q_i }|f|\leq 2^dA$.
   \end{itemize}
 \end{lem}

\begin{pr}
\label{pr4}
Let $\omega$ be a modulus of continuity
of order $n\in\mathbb{N}$ and $D\subset \mathbb{R}^d$
be a bounded domain. Then the seminorms
  \[ \|f\|_{\omega,D,p}=\sup_{Q\subset D} \inf_{P\in\mathcal{P}_n}\frac{1}{\omega(\ell)} \|f-P\|_{L^p(Q, dx/|Q|)}\]
are equivalent and define the same space $\mathcal{C}_{\om} (D)$ for $ 1\leq p\leq\infty.$
\end{pr}
\begin{proof}
It suffices to show that
\begin{equation}  \label{eq:eq9}
\sup_Q|f-\mathbb{P}_Q(f)|\lesssim\om(\ell)\|f\|_{\omega,D}
\end{equation}
with a constant independent of $Q$.
Here $\{\mathbb{P}_Q\}$ is the constructed in Sec.\,2.1 in a unified way family
of projectors from $L^1(Q,dx/|Q|)$ onto the polynomial subspace $\mathcal{P}_n(Q)$.
Let $C$ be the universal constant appearing in the corresponding near best approximation estimates.
Next, let $C'$ denote the norm of the projector
$\mathbb{P}_Q$ from $L^1(Q,dx/|Q|)$
onto the subspace $\mathcal{P}_n(Q)$ equipped with the uniform norm.

Select a cube $Q\subset D$.
We apply Lemma~\ref{lem6} to the function $|f-\mathbb{P}_Q(f)|$, $\|f\|_{\omega,D} =1$,
and with parameter $A=2 C\omega(\ell)$, where $\ell=\ell(Q)$.
Hence,
on the first step, we obtain a family $\{Q'_i\}$ of cubes $Q'_i\subset Q$
with the following properties:

$|f-\mathbb{P}_Q(f)|\leq 2 C\omega(\ell)$ a.e.\ on $Q\setminus\bigcup Q'_i;$

$|\mathbb{P}_{Q'_i}(f)-\mathbb{P}_Q(f)|=|\mathbb{P}_{Q'_i}(f-\mathbb{P}_Q(f))|
    \leq \frac{C'}{|Q'_i|} \int_{Q'_i} |f-\mathbb{P}_Q(f)|\leq 2^{d+1}C'C\omega(\ell)$;

$\sum|Q'_i|<\frac{1}{2 C\omega(\ell)} \int_Q |f-\mathbb{P}_Q(f)|\leq |Q|/2.$

Now, we apply the above construction with parameter $A=2C\omega(\ell')$ to the function
 $|f-\mathbb{P}_{Q'_i}(f)|$ for every cube $Q'$ from the family $\{Q'_i\}$.
Therefore,
on the second step, we obtain a family $\{Q''_i\}$ of cubes
$Q''_i \subset Q'$ with the following properties:
\begin{enumerate}
     \item[(i)]
     $|f-\mathbb{P}_{Q'}(f)|\leq 2C\omega(\ell')$ a.e.\ on $Q'\setminus\bigcup Q''_i;$
   \item[(ii)]
   $ |\mathbb{P}_{Q''_i}(f)-\mathbb{P}_{Q'}(f)|\!=\!|\mathbb{P}_{Q''_i}(f-\mathbb{P}_{Q'}(f))|
    \!\leq\! \frac{C'}{|Q''_i|} \int_{Q''_i}|f-\mathbb{P}_{Q'}(f)| \leq 2^{d+1}C'C\omega(\ell')$;
   \item[(iii)]
   $\sum|Q''_i|< \frac{1}{2 C\omega(\ell')}\int_{Q'}|f-\mathbb{P}_{Q'_i}(f)|\leq |Q'|/2.$
   \end{enumerate}
Summing the inequalities of type~(iii) over all cubes of the family $\{Q'_i\}$, we obtain
\[
 \sum|Q''|\leq\sum|Q'|/2\leq|Q|/4.
\]
Also, we have
\begin{align*}
  |f-\mathbb{P}_Q(f)|
&\leq|f-\mathbb{P}_{Q'_i}(f)|+|\mathbb{P}_{Q'_i}(f)-\mathbb{P}_{Q}(f)|
  \leq 2C\omega(\ell')+2^{d+1}C'C\omega(\ell) \\
 &\leq 2^{d+1}C'C(\omega(\ell')+\omega(\ell)) \quad\textrm{a.e.\ on}\ \bigcup{Q'}\setminus\bigcup{Q''}.
\end{align*}
Iterating the above procedure, we obtain
families of imbedded cubes $\{Q^k_j\}$, $k=0,\dots, m$, such that
every cube $Q^{k}_{i_k}$ is imbedded in an appropriate cube $Q^{k-1}_{i_{k-1}}$ and
\begin{equation}
\label{eq:eq10}
\sum_i|Q^k_i|<\frac{|Q|}{2^k}.
\end{equation}
Also, we have the estimate
\begin{equation}
\label{eq:eq11}
|f-\mathbb{P}_Q(f)|\leq2^{d+1} C'C\sum_{k=0}^{m-1}\omega(\ell(Q^k_{j_k}))
\:\textrm{a.e.\ on} \ \bigcup{Q^{m-1}}\setminus \bigcup{Q^{m}}
\end{equation}
for a sequence of embedded cubes
 $Q\supset Q'_{j_1}\supset\dots\supset Q^{m-1}_{i_{m-1}}$.

Let $m$ tend to infinity in (\ref {eq:eq11}).
Applying estimate (\ref{eq:eq10})
and property (\ref{eq: eq21}) for the function $\om$, we obtain
\[
\sum_{k=1}^{\infty}\omega(\ell(Q^k_{j_k}))
\lesssim \sum_{k=1}^{\infty}\omega\left(\frac{\ell }{2^{k}}\right)
\lesssim \int_{1}^\infty\omega(\ell/u)\frac{du}{u}
\lesssim\int_{0}^\ell\omega(t)\frac{dt}{t}
\lesssim\omega(\ell).
\]
Therefore,
$|f-\mathbb{P}_Q(f)|\lesssim \om (\ell)$
a.e.\ on $Q$
with a constant independent of $Q$.
The proof of the proposition is finished.
\end{proof}

\subsection {Estimates for polynomials of near best approximation}
Given a modulus of continuity $\om$ of order $n$, put
\begin{equation}\label{eq:eq23}
\xi(r)=\int_r^1\om(t)t^{-n-1}dt, \quad 0<r<1.
\end{equation}
Recall that the associated function is given by the equality
\[
\omt(t)=\om(t)/\max\{1,\xi(t)\}.
\]

We need the following auxiliary assertion.
\begin{lem}
\label{lem10}
Let $\om$ be a modulus of continuity of order $n$ and let
 $P_Q$ and $P_{2Q}$ be polyno\-mials of near best approximation for the function
      $f\in\mathcal{C}_{\omega}(\mathbb{R}^d)$ on $Q$ and $2Q$, respectively,
      $0< \ell(Q) < \frac{1}{2}$.
Consider the Taylor expansions
\begin{align*}
   P_Q(t) &=\sum_{k=0}^{n} A_{k,Q}(t-t_0)^k, \\
P_{2Q}(t) &=\sum_{k=0}^{n} A_{k,2Q}(t-t_0)^k
\end{align*}
with respect to $t_0 \in Q$. Then
\begin{equation}
\label{eq:eq213}
  |A_{k,2Q}-A_{k,Q}| \le C(n) \|f\| \om(\ell)\ell^{-|k|}, \,|k|=0,\dots,n.
\end{equation}
\end{lem}
\begin{proof}
The definition of the Zygmund space and the triangle inequality imply the estimate
\[\|P_{Q}-P_{2Q}\|_{L^\infty(Q)}\lesssim \om(\ell)\|f\|.
\]
Bernstein's inequality guarantees that
\[\|\partial^k P_{Q}-\partial^k P_{2Q}\|_{L^\infty(Q)} \le C(n)  \om(\ell)\ell^{-|k|}\|f\|\]
for all derivatives of order $k$, $|k|\leq n$. Hence,
\[
|A_{k,2Q}-A_{k,Q}|\le C(n)  \om(\ell)\ell^{-|k|}\|f\|,
\]
as required.
\end{proof}

\begin{lem}\label{lem9}
Let $\om$ be a modulus of continuity of order $n\in \mathbb{N}$ and $f\in\mathcal{C}_{\omega}(D)$.
Let $P_Q$ be a polynomial of near best approximation for $f$ on a cube $Q\subset D$
with center $t_0$ and side length $\ell<\frac{1}{2}$. Consider the Taylor expansion
\[P_Q(t)=\sum_{|k|=0}^{n} A_{k,Q}(t-t_0)^k\]
with respect to $t_0$, where $k=(k_1,\dots,k_d)$ is a multiindex,
$|k|=k_1+\dots+k_d$. Then
\begin{align*}
|A_{k,Q}|
 &\leq C \|f\|, \quad 0\leq |k|<n,  \\
|A_{k,Q}|
 &\leq C \|f\| \xi(\ell), \quad |k|=n,
\end{align*}
with a constant independent of $Q$.
   \end{lem}
\begin{proof}
Applying Proposition~\ref{pr3}, we extend $f$ up to $\widetilde f$ defined on $\mathbb{R}^d$.
Let
\[P_{2^iQ}(t)=\sum_{k=0}^{n} A_{k,2^iQ}(t-t_0)^k\]
be the Taylor decomposition with respect to $t_0 \in 2^iQ$
for the polynomial $P_{2^iQ}$ of near best approximation for the function $\widetilde f$ on the cube $2^iQ$.
Using telescopic sums, we have
\[|A_{k,Q}|\leq \sum_{i=0}^{N-1} |A_{k,2^iQ}-A_{k,2^{i+1}Q}|+|A_{k,2^{N}Q}|,
\]
where $N$ is the minimal natural number such that $2^{N}Q \supset D$.
Since $N\approx \log \frac{1}{\ell}$, Lemma~\ref{lem10} guarantees that
\[
|A_{k,Q}|\lesssim \left(1+\sum_{i=0}^{N} \om(2^i\ell) (2^i\ell)^{-|k|}\right) \|f\|
\lesssim\left(1+\int_{\ell}^1\frac{\om(t)dt}{t^{|k|+1}}\right)\|f\|.
\]
Since condition (\ref{eq:eq32}) holds for the function $\om$, we obtain
$|A_{k,Q}|\lesssim \|f\|$
for $|k|< n$.
Finally, by the definition of $\xi(x)$, we have
 \[|A_{k,Q}|\lesssim  \xi(\ell)\|f\|\]
for $|k|= n$.
The proof of the lemma is finished.
\end{proof}

Lemmas~\ref{lem1} and \ref{lem9} imply the following assertion.

\begin{cor}\label{cor9}
Let $\om$ be a modulus of continuity of order $n\in \mathbb{N}$,
$f\in\mathcal{C}_{\omega}(D)$ and let
$P_Q$ be a polynomial of near best approximation for $f$
on a cube $Q\subset D$ with $\ell<1/2$.
Then
 \begin{equation}
\label{eq:eq212}
  \|P_Q\|_{L^{\infty}(D)} \leq C \|f\| \xi(\ell)
\end{equation}
with a constant independent of $Q$.
\end{cor}

\subsection{Construction of extremal functions}
In this section, we find functions
$\varphi_e(x)\in \mathcal{C}_{\omega} (D)$, $e\in \mathbb{R}^d$, $|e|=1$,
and corresponding polynomials $P_{Q, e}$
of near best approximation with extremal properties
(cf.~\cite{Sj}).

Define the function
\[
\varphi(y)=\int_{|y|}^{1}\frac{\om(t)}{t^{n+1}}(t-y)^n dt, \quad y\in\mathbb{R}.
\]
For $x, e\in \mathbb{R}^d$, $|e|=1$, put
 $x_e=\langle x,e\rangle$, where $\langle\cdot, \cdot \rangle$ denotes
the scalar product in $\mathbb{R}^d$.
Consider the function
 \begin{equation}
\label{eq:eq300}
 \varphi_e(x)=\varphi(x_e), \quad x\in \mathbb{R}^d.
  \end{equation}

For a positive parameter $\gamma$, the function
\[
P_\gamma(y)=\int_{\gamma}^{1}\frac{\om(t)}{t^{n+1}}(t-y)^n dt, \quad y\in\mathbb{R},
\]
is a polynomial in $y$.
Given a cube $Q$, put $\gamma=\max_{x\in Q}|x_e|$ and define the following polynomial:
\begin{equation}
\label{eq:eq301}
P_{Q,e}(x)=P_{\gamma}(x_e)\in \mathcal{P}_n(\mathbb{R}^d).
\end{equation}

\begin{lem}
\label{lem12}
Let $\varphi_e(x)$ and $P_{Q,e}(x)$ be defined by \eqref{eq:eq300} and \eqref{eq:eq301}, respectively.
\begin{itemize}
\item[(i)]
  The norms $\|\varphi_e(x)\|_{\mathcal{C}_{\om}(\mathbb{R}^d)}$ are uniformly bounded for $|e|=1$;
  if $\ell(Q)<\frac{1}{2}$, then
   $P_{Q,e}(x)$ is a polynomial of near best approximation for $\varphi_e$ on the cube $Q$.
\item[(ii)]
If $Q$ is a cube centered at the origin, $\gamma(Q)<1$ and
\[
P_{Q,e}(x)=\sum_{k=0}^{n} A_{k,Q,e} \langle x,e\rangle^k
\]
is a homogeneous expansion, then
\[
 |A_{n,Q,e}|\geq C \xi(\ell)
\]
with a constant $C>0$ independent of $Q$ and $e$, $|e|=1$.
\end{itemize}
 \end{lem}
 \begin{proof}
Firstly, we prove property~(i).
Let $\ell(Q)<1/2$.
We have
\[\sup_{x\in Q} |\varphi_e(x)-P_{Q,e}(x)|\leq \sup_{x\in Q} \int_{|x_e|}^{\gamma}\frac{\om(t)}{t^{n+1}}(t-x_e)^n dt.\]

If $\ell<\frac{1}{2} \gamma$,
then property~(\ref{eq:eq42}) of the function $\om$ guarantees that
\[
 \int_{|x_e|}^{\gamma}\frac{\om(t)}{t^{n+1}}(t-x_e)^n dt
 \lesssim \frac{\om(|x_e|)}{|x_e|^{n+1}}\ell^{n+1}
 \lesssim \om(\ell).
\]

If $\ell\geq \frac{1}{2} \gamma$, then $\gamma\approx \ell$.
We have $t-x_e\leq t$ for $x_e\geq 0$, and
  $t-x_e\leq t+|x_e|\leq 2t$ for $x_e < 0$ and $|x_e|\le t$.
 Thus, by (\ref{eq:eq32}), we obtain
\[
 \int_{|x_e|}^{\gamma}\frac{\om(t)}{t^{n+1}}(t-x_e)^n dt
 \lesssim\int_{|x_e|}^{\gamma}\frac{\om(t)}{t^{n+1}}t^n dt
 \lesssim\frac{\om(\gamma)}{\gamma^{n-\varepsilon}}\int_{|x_e|}^{\gamma}{t^{n-1-\varepsilon}} dt
 \lesssim \om(\gamma) \lesssim \om(\ell).
\]
So, part~(i) is proven.

To prove part~(ii), observe that
\[  A_{n,Q,e} = (-1)^n  \int_{\gamma}^{1}\frac{\om(t)}{t^{n+1}} dt.
\]
Since $\gamma\approx \ell$
for the cube $Q$ under consideration, we have
\[  |A_{n,Q,e}|\geq C \xi(\ell)
\]
with a constant $C>0$ independent of $Q$ and $e$, $|e|=1$.
The proof of the lemma is finished.
\end{proof}

\section{Proof of Theorem~\ref{thm3}}
\subsection {Main auxiliary construction}
Fix a function  $f\in\mathcal{C}_{\omega}(D)$.
Put
\[
\|f\| =\|f\|_{\om,D}+\|f\|_{L^1(D)}.
\]
Consider an arbitrary cube $Q$
such that $2Q \subset D$. Let $x_0$ denote the center of $Q$,  $\ell=\ell(Q)$.

Let $P_Q$ be a polynomial of \emph{near best} approximation for $f$ in the cube $Q$.
Consider the following auxiliary functions
(see, for example, \cite{DV, H, V1} for similar arguments):
\[
\begin{aligned}
  f_1 &=P_Q \chi_D, \\
  f_2 &=(f-P_Q) \chi_{2Q}, \\
  f_3 &=(f-P_Q) \chi_{D\backslash 2Q}.
\end{aligned}
\]
Observe that $f = f_1 + f_2 + f_3$.
The following lemma shows how to properly handle the functions $T_D f_2$ and $T_D f_3$.

\begin{lem}\label{lem11}
There exist polynomials $ P_{k,Q}$, $k=2,3$, such that
        \[\frac{1}{|Q|}\int_{Q}|T_D f_k-P_{k,Q}|dx\leq C\omega (\ell )\|f\|\]
with a constant $C>0$ independent of $Q$.
\end{lem}

\begin{proof}[Proof of Lemma~\ref{lem11} for $k=2$]
Put $P_{2,Q}=0$.
By H\"older's inequality, we have
\[
I_2=\frac{1}{|Q|}\int_Q|T_Df_2|dx
\le \left(\frac{1}{|Q|}\int_Q|T_Df_2|^2dx\right)^{1/2}.
\]
The operator $T_D$ is known to be bounded on $L^2$ (see \cite[Ch.\,2]{S}).
Therefore,
\[
\begin{aligned}
I_2
&\lesssim\left(\frac{1}{|Q|}\int_{2Q}|f_2|^2dx\right)^{1/2}
 = \left(\frac{1}{|Q|}\int_{2Q}|f-P_Q|^2dx\right)^{1/2} \\
&\lesssim
\left(\frac{1}{|Q|}\int_{2Q}|f-P_{2Q}|^2dx\right)^{1/2}
+\left(\frac{1}{|Q|}\int_{2Q}|P_Q-P_{2Q}|^2dx\right)^{1/2}.
\end{aligned}
\]
Now, observe that the first summand is estimated by  $C\omega(\ell)\|f\|_{\om, D}$.
Indeed, the proof of Proposition~\ref{pr4} allows to replace the $L^2$-norm by the $L^1$-norm,
thus, it remains to apply Definition~\ref{df22} for the polynomial $P_{2Q}$.
Next, Lemma~\ref{lem3} guarantees that the second summand is also estimated by
$C\omega(\ell)\|f\|_{\om, D}$.
Hence, the proof of the lemma for $k=2$ is finished.
\end{proof}

\begin{proof}[Proof of Lemma~\ref{lem11} for $k=3$]
To estimate the oscillation
\[
I_3=\frac{1}{|Q|}\int_Q|T_Df_3-P_{3,Q}| \,dx,
\]
we define $P_{3,Q}$ as the image of $f_3$ under the action of a special integral operator with a polynomial kernel.
Namely,
consider the Taylor polynomial of order $n$ for the kernel $K(x) ={\Omega(x)}{|x|^{-d}}$ of the operator $T$ at $y$, $y\neq 0$:
\[
\mathcal{T}K(y, h)=K(y)+ (\nabla_y K )(h) +\dots
+ \frac{\nabla_y^{n} K}{n!}(h), \quad h\in\mathbb{R}^d,\ |h| < |y|/2,
\]
where $\nabla_y^j K$ denotes the differential of order $j$ for $K$ at $y$.
Recall that $x_0$ is the center of $Q$.
Define the polynomial $P_{3,Q}$ as follows:
\[
P_{3,Q}(x)=\int_{D\setminus 2Q} \mathcal{T}K(x_0-u, x- x_0) f_3(u)\,du.
\]

The kernel $K$ is $C^{n+1}$-smooth, thus,
\[
\left| \nabla_y^j K \right| \lesssim |y|^{-d-j}, \quad j=0, 1, \dots, n+1,\quad y\neq 0.
\]
Hence, for $u\in \mathbb{R}^d\setminus 2Q$ and $x\in Q$,
the remainder in the Taylor formula is estimated as follows:
\[
\begin{aligned}
 |K(x-u)-\mathcal{T}K(x_0-u, x-x_0)|
&\le C\sup _{t\in Q,\, u\notin 2Q}|\nabla^{n+1}_t K(t-u)|\frac{|x-x_0|^{n+1}}{(n+1)!} \\
&\le C \frac{|x-x_0|^{n+1}}{|u-x_0|^{n+1+d}},
\end{aligned}
\]
where the constant $C>0$
does not depend on $u$, $x$, $x_0$ and $Q$.
Applying the above estimate, we have
\[
I_3
\le \frac{C}{|Q|}\int_Q dx \int_{D\setminus 2Q}
  \frac{|x-x_0|^{n+1}}{|u-x_0|^{n+1+d}}|f-P_Q|(u) \,du
\lesssim \ell^{n+1}\int_{D\backslash 2Q} \frac{|f-P_Q|(u)}{|u-x_0|^{n+1+d}} \,du.
\]
Now, define $\widetilde{f}$
by means of formula~(\ref{eq:eq222}).
We have
 $$
 I_3\lesssim \ell^{n+1}\int_{\mathbb{R}^d\setminus 2Q}
 \frac{|\widetilde{f}-P_Q|(u)}{|u-x_0|^{n+1+d}} \,du.
 $$
Put $Q_k= 2^{k+1}Q\setminus 2^{k}Q$ and rewrite
the above estimate as follows:
$$
I_3\lesssim\ell^{n+1}\sum_{k=1}^\infty
\frac{1}{(\ell2^k)^{n+1+d}}\int_{Q_k} |\widetilde{f}-P_Q|(u) \,du.
$$
Using telescoping summation, we obtain
\[
I_3\lesssim\ell^{n+1}\sum_{k=1}^\infty \frac{1}{(\ell2^k)^{n+1+d}}\left(\int_{Q_k} |\widetilde{f}-P_{2^{k+1}Q}|(u)du+\sum_{s=0}^{k} \int_{Q_k}|P_{2^sQ}-P_{2^{s+1}Q}|(u)du \right),
\]
where $P_{2^sQ}:=\mathbb{P}_{2^sQ}\widetilde{f}$ is a polynomial
of near best approximation for $\widetilde{f}$ on the cube $2^sQ$, $s=0,1,\dots, k$, $k\in\Nbb$.
Observe that
\begin{align*}
  \frac{1}{|2^{k+1}Q|} \int_{Q_k}|P_{2^sQ}
  &-P_{2^{s+1}Q}|(u)du
   \le \frac{1}{|2^{k+1}Q|} \int_{2^{k+1}Q}|P_{2^sQ}-P_{2^{s+1}Q}|(u)du  \\
  &\lesssim \left(\frac{\ell 2^{k+1}}{\ell 2^{s+1}}\right)^n
   \frac{1}{|2^{k+1}Q|} \int_{2^{s+1}Q}|P_{2^sQ}-P_{2^{s+1}Q}|(u)du  \\
  &\lesssim \left(\frac{2^{k}}{2^{s}}\right)^n \om(\ell 2^{s+1})\|f\|_\om
\end{align*}
by Lemmas~\ref{lem2} and \ref{lem3}, respectively.
Applying the above inequality, we estimate $I_3$ and obtain
\[
 I_3
\lesssim \ell^{n+1} \sum_{k=1}^\infty \frac{1}{(\ell 2^k)^{n+1+d}}
 \sum_{s=0}^k \frac{(\ell 2^{k+1})^{d} 2^{kn}}{2^{sn}} \omega(2^s\ell)\|f\|_\omega
\lesssim\sum_{k=1}^\infty \frac{1}{2^k}\sum_{s=0}^k \frac{1}{2^{sn}}
 \omega(2^s\ell)\|f\|_\omega.
\]
Changing the summation order, we have
\[
 I_3
\lesssim \|f\|_\omega \sum_{s=0}^\infty\frac{1}{2^{s(n+1)}}\omega(2^s\ell)
\lesssim \|f\|_\omega \int_1^\infty \frac{\omega(t\ell)}{t^{n+2}}dt.
\]
Finally, changing the variable of integration and applying property~(\ref{eq: eq22})
from Lemma~\ref{lem21}, we obtain the required inequality
$$
I_3\le C\omega(\ell)\|f\|_\omega.
$$
The proof of the lemma is finished.
\end{proof}

\subsection{Proof of Theorem~\ref{thm3}: sufficiency}
Assume that properties (i) and (ii) from Theorem~\ref{thm3} hold.
Let
$f\in\mathcal{C}_{\omega}(D)$ and
$Q$ be a cube such that
$2Q \subset D$.
We have $\mathcal{C}_{\om}(D) =\mathcal{C}^{int}_{\om}(D)$ by Proposition~\ref{pr3}.
Hence, to prove the required implication, it suffices to verity the following property:
there exists a polynomial $S_Q \in \mathcal{P}_n$ such that
\begin{equation}\label{eq:eq20}
   I=\frac{1}{|Q|}\int_{Q}|T_D f-S_Q|dx \le C \omega (\ell )\|f\|,
\end{equation}
where the constant $C>0$ does not depend on $Q$.

Let $P_Q$ be a polynomial of near best approximation for $f$ in the cube $Q$ under considera\-tion.
We write the Taylor expansion with respect to the center $x_0\in Q$ as follows:
\begin{align*}
  P_Q(x)=\sum_{|k|=0}^{n} A_{k,Q}(x-x_0)^k
&=\sum_{|k|=0}^{n-1} A_{k,Q}(x-x_0)^k + \sum_{|k|=n} A_{k,Q}(x-x_0)^k \\
& :=P_{n-1}(x)+P_n(x),
\end{align*}
where $k=(k_1,\dots,k_d)$ denotes a multiindex, $|k|=k_1+\dots+k_d$.

Since condition~(i) from Theorem~\ref{thm3} holds,
there exists a polynomial
$S_{1, Q} \in \mathcal{P}_n$ such that
 \[
 J_1:= \frac{1}{|Q|}\int_{Q }|T_D (\chi_D P_{n-1}) - S_{1, Q}|\,dx
 \lesssim\ \om(\ell) \|P_{n-1}\|_{L^\infty(D)}
 \]
with a constant independent of $Q$. Lemmas~\ref{lem1} and \ref{lem9} guarantee that
\[
J_1 \lesssim \om(\ell) \|f\|.
\]
By condition~(ii) from Theorem~\ref{thm3}, there exists a polynomial
 $S_{2, Q} \in \mathcal{P}_n$ such that
 \[
 J_2:= \frac{1}{|Q|}\int_{Q }|T_D (\chi_D P_n) - S_{2, Q}|\,dx
 \lesssim\|P_n\|_{L^\infty(D)} \omt( \ell)
 \]
 with a constant independent of $Q$. Corollary~\ref{cor9} guarantees that
\[
J_2 \lesssim\xi(\ell) \omt( \ell)\|f\| \lesssim \om(\ell)\|f\|.
\]
Combining the estimates obtained for $J_1$ and $J_2$, we have
 \[
 \frac{1}{|Q|}\int_{Q}| T_D (\chi_D P_Q) -S_{1, Q}-S_{2,Q}|\,dx
 \lesssim \om(\ell) \|f\|.
 \]

Recall that $\chi_D P_Q = f_1$ in the notation of Lemma~\ref{lem11}.
Therefore, the estimate obtained and Lemma~\ref{lem11} imply
the required property (\ref{eq:eq20}).
The proof of sufficiency is finished.

\subsection{Proof of Theorem~\ref{thm3}: necessity}
Since $\mathcal{P}_n(D) \subset \mathcal{C}_{\omega} (D)$,
the necessity
is clear for a Dini regular modulus of continuity $\om$.
Indeed, in this case, $\om(t)\approx \omt(t)$
and condition~(i) from Theorem~\ref{thm3} implies condition~(ii).
We have a standard symmetric T(P) theorem without additional condition~(ii).

Now, assume that the modulus of continuity $\om$ is not Dini regular.
We have to prove that condition~(ii) holds.
Observe that the functions $\om(t)$ and $\omt(t)$ are equivalent for $\frac{1}{2}\le t <\infty$,
since $\om(t)= \omt(t)$ for $t\ge 1$.
Thus, to verify condition~(ii) for $Q$, we may assume that $\ell(Q)<\frac{1}{2}$.

Consider the following family of shifts for the extremal function defined by (\ref{eq:eq300}):
\[\varphi_{e,x_0} (x)=\varphi_e(x-x_0),\quad  x_0\in D.
\]
Since $D$ is a Lipschitz domain, the properties of
$\varphi_{e,x_0}\chi_D$
are similar to those of $\varphi_e$. Namely,
   $\varphi_{e,x_0} \chi_D\in\mathcal{C}_{\om}(D)$
and the corresponding norms in the space $\mathcal{C}_{\om}(D)$
are bounded, uniformly with respect to $x_0\in D$ and $e$, $|e|=1$,
by a constant depending only on the Lipschitz constants of the domain $D$.
For every function $\varphi_{e,x_0} \chi_D$,
we choose a polynomial of near best approximation with the help of (\ref{eq:eq301}) as follows:
\[
   P_{e,x_0,Q}(x)=P_{e,Q}(x-x_0)=\sum_{k=0}^{n} A_{k,Q,e} \langle x-x_0, e\rangle^k :=
   P_{n-1}(x)+ A_{n,Q,t}\langle x-x_0, e\rangle^n
\]
with coefficients independent of the point $x_0\in D$.
  Put $f=\varphi_{e,x_0}$
  and $f_1=P_{e,x_0,Q}\chi_D$.
By assumption, the operator $T_D$ is bounded on $\mathcal{C}_{\om}(D)$.
Therefore, by Lemma~\ref{lem11} and the triangle inequality, there exists
a polynomial $S_Q\in \mathcal{P}_n$ such that
\[
 \frac{1}{|Q|}\int_{Q }|T_D f_1-S_Q|dx\lesssim \om(\ell) \|\varphi_{e,x_0}\|.
\]
Since $\chi_D P_{n-1}\in \mathcal{P}_n(D) \subset \mathcal{C}_{\omega} (D)$, there exists
a polynomial $S_Q'\in \mathcal{P}_n$ such that
\[
 \frac{1}{|Q|}\int_{Q }|T_D (\chi_D P_{n-1})-S_Q^\prime|dx
 \lesssim \om(\ell) \|P_{n-1}\|
 \lesssim \om(\ell) \|\varphi_{e,x_0}\|
\]
by Lemmas~\ref{lem1} and \ref{lem9}. Therefore, the triangle inequality guarantees that
\[
 \frac{1}{|Q|}\int_{Q }|T_D  (\chi_D A_{n,Q,e} \langle x-x_0, e\rangle^n) -(S_Q-S_Q')(x)|dx\lesssim \om(\ell) \|\varphi_{e,x_0}\|
 \lesssim \om(\ell) \|\varphi\|
\]
with a constant depending only on the Lipschitz constants of the domain $D$.
Next, put $R_Q=(S_Q-S_Q')/A_{k,Q,e}$
and rewrite the above inequality
as follows:
\[|A_{n,Q,e}|\frac{1}{|Q|}\int_{Q }|T_D (\chi_D \langle x-x_0, e\rangle^n) -R_Q(x)|dx\lesssim \om(\ell) \|\varphi\|.\]
By Lemma \ref{lem12}, we have
\begin{equation}\label{eq:eq400}
\frac{1}{|Q|}\int_{Q }|T_D (\chi_D \langle x-x_0, e\rangle^n) - R_Q(x)|dx\lesssim
 \frac{\om(\ell)}{\xi(\ell)} \|\varphi\|
\lesssim \omt(\ell)\|\varphi\|
\end{equation}
uniformly with respect to $x_0 \in D$ and $e$, $|e|=1$.
It remains to observe that property (\ref{eq:eq400}) implies condition~(ii) from Theorem~\ref{thm3}.
The proof of necessity is finished.


\bibliographystyle{amsplain}

\begin{thebibliography}{10}

\bibitem{An}
D.~S. Anikonov, \emph{On the boundedness of a singular integral operator in the
  space {$C\sp{\alpha }(\overline G)$}}, Math. USSR-Sb. \textbf{33} (1977),
  no.~4, 447--464. \MR{0487594}

\bibitem{BCFST13}
J.~J. Betancor, R.~Crescimbeni, J.~C. Fari\~{n}a, P.~R. Stinga, and J.~L.
  Torrea, \emph{A {$T1$} criterion for {H}ermite-{C}alder\'{o}n-{Z}ygmund
  operators on the {$BMO_H(\mathbb{R}^n)$} space and applications}, Ann. Sc.
  Norm. Super. Pisa Cl. Sci. (5) \textbf{12} (2013), no.~1, 157--187.
  \MR{3088440}

\bibitem{C}
S.~Campanato, \emph{Propriet\`a di h\"{o}lderianit\`a di alcune classi di
  funzioni}, Ann. Scuola Norm. Sup. Pisa (3) \textbf{17} (1963), 175--188.
  \MR{0156188}

\bibitem{DVS}
R.~A. DeVore and R.~C. Sharpley, \emph{Besov spaces on domains in {${\bf
  R}^d$}}, Trans. Amer. Math. Soc. \textbf{335} (1993), no.~2, 843--864.
  \MR{1152321}

\bibitem{DV}
E.~Doubtsov and A.~V. Vasin, \emph{Restricted {B}eurling transforms on
  {C}ampanato spaces}, Complex Var. Elliptic Equ. \textbf{62} (2017), no.~3,
  333--346. \MR{3598981}

\bibitem{H}
T.~Hansson, \emph{On {H}ardy spaces in complex ellipsoids}, Ann. Inst. Fourier
  (Grenoble) \textbf{49} (1999), no.~5, 1477--1501. \MR{1723824}

\bibitem{Ja2}
S.~Janson, \emph{Generalizations of {L}ipschitz spaces and an application to
  {H}ardy spaces and bounded mean oscillation}, Duke Math. J. \textbf{47}
  (1980), no.~4, 959--982. \MR{596123}

\bibitem{Jo}
P.~W. Jones, \emph{Extension theorems for {BMO}}, Indiana Univ. Math. J.
  \textbf{29} (1980), no.~1, 41--66. \MR{554817}

\bibitem{KK}
S.~Kislyakov and N.~Kruglyak, \emph{Extremal problems in interpolation theory,
  {W}hitney-{B}e\-sicovitch coverings, and singular integrals}, Instytut
  Matematyczny Polskiej Akademii Nauk. Monografie Matematyczne (New Series)
  [Mathematics Institute of the Polish Academy of Sciences. Mathematical
  Monographs (New Series)], vol.~74, Birkh\"{a}user/Springer Basel AG, Basel,
  2013. \MR{2975808}

\bibitem{MOV}
J.~Mateu, J.~Orobitg, and J.~Verdera, \emph{Extra cancellation of even
  {C}alder\'{o}n-{Z}ygmund operators and quasiconformal mappings}, J. Math.
  Pures Appl. (9) \textbf{91} (2009), no.~4, 402--431. \MR{2518005}

\bibitem{M}
N.~G. Meyers, \emph{Mean oscillation over cubes and {H}\"{o}lder continuity},
  Proc. Amer. Math. Soc. \textbf{15} (1964), 717--721. \MR{0168712}

\bibitem{PT}
M.~Prats and X.~Tolsa, \emph{A {$T(P)$} theorem for {S}obolev spaces on
  domains}, J. Funct. Anal. \textbf{268} (2015), no.~10, 2946--2989.
  \MR{3331790}

\bibitem{Sj}
T.~Sj\"{o}din, \emph{On properties of functions with conditions on their mean
  oscillation over cubes}, Ark. Mat. \textbf{20} (1982), no.~2, 275--291.
  \MR{686176}

\bibitem{S}
E.~M. Stein, \emph{Singular integrals and differentiability properties of
  functions}, Princeton Mathematical Series, No. 30, Princeton University
  Press, Princeton, N.J., 1970. \MR{0290095}

\bibitem{T05}
X.~Tolsa, \emph{Bilipschitz maps, analytic capacity, and the {C}auchy integral},
Ann. of Math. (2) \textbf{162} (2005), no.~3, 1243--1304. \MR{2179730}

\bibitem{V1}
A.~V. Vasin, \emph{A {T}1 theorem and {C}alder\'{o}n-{Z}ygmund operators in
  {C}ampanato spaces on domains}, Math. Nachr. \textbf{292} (2019), no.~6,
  1392--1407. \MR{3959468}

\end{thebibliography}

\end{document}